\documentclass[11pt]{article}
\usepackage{amsmath,amsthm,amsfonts,amssymb}
\usepackage{epsfig}
\usepackage[usenames]{color}
\usepackage{verbatim}

\parskip 	\bigskipamount

\newtheorem{theorem}{Theorem}[section]
\newtheorem{definition}[theorem]{Definition}

\numberwithin{equation}{section}
\newtheorem{lemma}[theorem]{Lemma}
\newtheorem{proposition}[theorem]{Proposition}
\newtheorem{corollary}[theorem]{Corollary}

\newtheorem{remark}[theorem]{Remark}
\newtheorem{example}[theorem]{Example}

\numberwithin{equation}{section}

\def\Q{\mathbb{Q}}
\def\R{\mathbb{R}}

\def\P{\mathbb{P}}
\def\E{\mathbb{E}}

\renewcommand{\phi}{\varphi}
\renewcommand{\epsilon}{\varepsilon}

\newcommand{\1}{{\text{\Large $\mathfrak 1$}}}

\renewcommand{\emptyset}{\varnothing}

\newcommand\be{\begin{equation}}
\newcommand\ee{\end{equation}}

\def\p{\Psi}

\begin{document}

\title{\bf Brownian motion with variable drift: 0-1 laws, hitting probabilities and Hausdorff dimension}

\author{
Yuval Peres\thanks{Microsoft Research, Redmond, Washington, USA; peres@microsoft.com} \and Perla Sousi\thanks{University of Cambridge, Cambridge, UK;   p.sousi@statslab.cam.ac.uk}
}
\date{}
\maketitle

\begin{abstract}
By the Cameron--Martin theorem, if a function $f$ is in the Dirichlet space $D$, then $B+f$ has the same a.s.\ properties as standard Brownian motion, $B$. In this paper we examine properties of $B+f$ when $f \notin D$.
We start by establishing a general 0-1 law, which in particular implies that for any fixed $f$, the Hausdorff dimension of the image and the graph of $B+f$ are constants a.s. (This 0-1 law applies to any L\'evy process.) Then we show that if the function $f$ is H\"older$(\frac{1}{2})$, then $B+f$ is {\em intersection equivalent\/} to $B$. Moreover, $B+f$  has double points a.s.\ in dimensions $d\le 3$, while in $d\ge 4$ it does not. We also give examples of functions which are H\"older with exponent less than $\frac{1}{2}$, that yield double points in dimensions greater than 4. Finally, we show that for $d \ge 2$,
the Hausdorff dimension of the image of $B+f$ is a.s.\ at least the maximum of 2 and the dimension of the image of $f$.
\newline
MSC 2010 \emph{subject classifications.} Primary  60J65 ;  
Secondary 60F20  , 60J45.    
\end{abstract}

\section{Introduction}
By the Cameron--Martin theorem, if a function $f$ is in the Dirichlet space,
\[
D[0,1] = \left\{ F \in C[0,1]: \exists f \in \mathbf{L}^2[0,1] \text{ such that } F(t) = \int_{0}^{t} f(s) ds, \forall t \in [0,1] \right\},
\]
then $B+f$ has the same a.s.\ properties as standard Brownian motion, $B$.
When $f$ is not in this space, then the laws of $B$ and $B+f$  are singular, i.e., there exists a
Borel set $A$ in $C[0,1]$ such that $\mathbb{L}_{B}(A)=0$ and $\mathbb{L}_{B+f}(A^c)=0$, where $\mathbb{L}_{B}$ and  $\mathbb{L}_{B+f}$ are the laws of $B$ and $B+f$, respectively. Thus, if $f \notin D[0,1]$, then there   exists an almost sure property of $B$ which is not an almost sure property of $B+f$.

In this paper we consider functions $f \notin D[0,1]$, and investigate whether some specific a.s.\ properties of Brownian motion transfer also to $B+f$.

In Section \ref{01} we establish  a   zero-one law for  random set functions, Theorem \ref{0-1}.
Here we state some  special cases of this law. For a function $f: [0,1] \to \R^d$, denote its graph restricted to $A \subset [0,1]$ by
\[
G(f)(A) = \{(t,f(t)): t \in A \}.
\]
When $A=[0,1]$, we abbreviate $G(f)([0,1])$ as $G(f)$.

For concreteness, we state the following theorem for Brownian motion. Parts (\textbf{a}) and (\textbf{c}) apply to any c\`adl\`ag process with independent increments.

\begin{theorem}\label{bmdrift}
Let $(B_t, 0\le t \le 1)$ be a standard Brownian motion in $\R^d$ and let $f:[0,1] \to \R^d$ be a continuous function.
\begin{description}
  \item[(a)] Let $A$ be a closed set in $[0,1]$. Then $\P(\mathcal{L}(B+f)(A)>0) \in \{0,1\}$,
           where $\mathcal{L}$ stands for Lebesgue measure and $(B+f)(A)$ is the image of $A$ under $B+f$.
  \item[(b)] Let $A\subset [0,1]$ be closed. Then $\P(((B+f)(A))^\circ \neq \emptyset) \in \{0,1\}$, where
           $D^\circ$ stands for the interior of $D$.
  \item[(c)] $\dim (B+f)[0,1]=c$ a.s.\ and  $\dim G(B+f)[0,1]=c'$ a.s.\, where $c$  and $c'$ are two positive constants and $\dim$ is the Hausdorff dimension.
\end{description}
\end{theorem}


\begin{remark}
\rm{
We know that in dimensions $d \geq 2$, Brownian motion without drift has 0 volume. A natural question that arises is whether there exist functions $f$ such that the Lebesgue measure of $(B+f)[0,1]$ is positive. Such functions have been constructed
by Graversen \cite{Graversen}. For all $d\geq 2$ and all $\alpha< \frac{1}{d}$, Antunovi\'c, Peres and Vermesi \cite{Tonci} construct an $\alpha$-H\"older continuous function, $f$, which makes the image of $B+f$ cover an open set a.s. Related results for L\'evy processes are in \cite{SE}.
}
\end{remark}

In Section \ref{counterex} we discuss an example of an event which seems similar to those in Theorem~\ref{bmdrift}, yet does not follow a 0-1 law. This example answers a question of Itai Benjamini (personal communication).

Before stating the other results we obtain in this paper, we first recall some definitions.

\begin{definition}\rm{
Let $X = (X_t, t\geq 0)$ be a stochastic process.
A closed set $A\subset \R^d$ is called \textbf{nonpolar} for $X$, if $\P_x(X \text{ hits } A)>0$, for all $x$. Otherwise, it is called \textbf{polar}.
}
\end{definition}

\begin{definition}
\rm{
Let $B$ be a standard Brownian motion in 2 dimensions. We denote by $B^{\lambda}$, the Brownian motion killed after an independent Exponential time of parameter $\lambda$.
}
\end{definition}

In Section \ref{hitting} we obtain results about the hitting probabilities of $B+f$, when $f$ is a H\"older$\left(\frac{1}{2}\right)$ continuous function. We recall the definition of H\"older($\alpha$) continuous function:
\[
\exists K: \forall x,y \quad |f(x)-f(y)|\leq K|x-y|^\alpha.
\]
We call $K$ the H\"older($\alpha$) constant of $f$.
We will prove the following theorem:

\begin{theorem}\label{interequi}
Let $(B_t)_t$ be a standard Brownian motion in $d\geq 2$ dimensions and $f$ a H\"older$\left(\frac{1}{2}\right)$ continuous function $f: \R_+ \to \R^d$ with H\"older constant $K$.
\begin{description}
\item[(a)]
\quad If $d \geq 3$, then $B+f$ and $B$ are intersection equivalent, in the sense that
there exist positive constants $c_1$ and $c_2$ depending only on $d$ and $K$ such that for all $x \in \R^d$ and all closed sets $A\subset \R^d$, we have
\[
c_1 \P_x(B \text{\rm{ hits }} A) \leq \P_x(B+f \text{\rm{ hits }} A) \leq c_2 \P_x(B \text{\rm{ hits }} A),
\]
where the notation $\P_x(B+f \text{\rm{ hits }} A)$ means that $B_0+f(0)=x$.
In particular, if a closed set $\Lambda$ is polar for $B$, then it is also polar for $B+f$.
\item[(b)] \quad If $d=2$, then for any bounded open set $U$, there exist positive constants $c_1$ and $c_2$ depending on $K$, on $U$ and $\lambda$ such that for all $x \in U$  and all closed sets $A \subset U$, we have that
    \[
    c_1 \P_x(B^{\lambda} \text{\rm{ hits }} A) \leq \P_x((B+f)^{\lambda} \text{\rm{ hits }} A) \leq c_2 \P_x(B^{\lambda} \text{\rm{ hits }} A).
    \]
Also if a closed set $A$ is nonpolar for $B$, then it is also nonpolar for $B+f$ and
$\P_x(B+f \text{\rm{ hits }} A) = 1$, for all $x$. Finally, if $A$, a closed set, is polar for $B$, then it is also polar for $B+f$.
\end{description}
\end{theorem}

Next, in Section \ref{double}, using the ``intersection equivalence'' given in Theorem~\ref{interequi}, we show that when the function $f$ is again a H\"older($\frac{1}{2}$) continuous function, then  double points for $B+f$ exist only in dimension $d \leq 3$, just like in the Brownian motion case. We will prove the following theorem:

\begin{theorem} \label{doublepoints}
Let $(B_t, 0 \le t \le 1)$ be a standard Brownian motion in $d$ dimensions and let $f$ be a H\"older$\left(\frac{1}{2}\right)$ function.
\begin{description}
  \item[(a)] If $d\le 3$, then, almost surely, $B+f$ has double points.
  \item[(b)] If $d\ge 4$, then, almost surely, $B+f$ has no double points.
\end{description}
\end{theorem}
In Theorem \ref{fractional} we show that in dimension $d \ge 4$, there exist $f \notin$ H\"older($\frac{1}{2}$) such that $B+f$ has double points.

\begin{remark}\rm{
1. The space of H\"older($\alpha$) continuous functions functions is much larger than the Dirichlet space, $D[0,1]$. Indeed, for any $\alpha \in (0,1)$, most H\"older($\alpha$) continuous functions (in the sense of Baire category) are nowhere differentiable.
\newline
2. The value $\alpha=1/2$ is also the critical H\"older exponent for other properties of Brownian motion with drift, such as positive area in 2 dimensions (see \cite{Graversen}) and isolated zeros in 1 dimension (see \cite{ABP}).
}
\end{remark}

Finally, in Section \ref{hausd}, we study the Hausdorff dimension of the image and graph of $B+f$, when $f$ is a continuous function.

From Theorem \ref{bmdrift} (part c), we have that the Hausdorff dimension of the image and the graph are constants a.s. In Section \ref{hausd} we obtain upper and lower bounds for these constants.

Recall from McKean's theorem (see for instance \cite[Theorem~4.33]{MP}) that almost surely
\[
\dim B(A) = (2\dim A) \wedge d.
\]
In the following Theorem we show that adding a continuous drift cannot decrease the Hausdorff dimension of the image.
\begin{theorem}
\label{upperbound}
Let $f: [0,1] \to \R^d$, $d \geq 1$, be a continuous function and let $(B_t, 0 \le t \le 1)$ be a standard Brownian motion in $d$ dimensions. Let $A$ be a closed subset of $[0,1]$. Then
\[
\dim (B+f)(A) \geq \max\{(2\dim A) \wedge d , \dim f(A)\}.
\]
\end{theorem}

Our next result concerns the dimension of the graph.
\begin{theorem}
\label{graphdim}
Let $f$ and $B$ be as in Theorem \ref{upperbound}.
\begin{description}
  \item[(a)] If $d = 1$, then $\dim G(B+f) \geq \max\{\frac{3}{2}, \dim G(f) \}$.
  \item[(b)] If $d \geq 2$, then $\dim G(B+f) \geq \max\{2, \dim G(f) \}$.
\end{description}
\end{theorem}

\section{0-1 laws} \label{01}
In this Section we first prove the general zero one law announced in the Introduction and then apply it to give the proof of Theorem~\ref{bmdrift}.
\subsection{The theorems}

\begin{theorem}\label{0-1}
Let $(\Omega, \mathcal{F},P)$ be a probability space.
Let $\Psi$ be a random set function, i.e. $\Psi: \Omega \times \{\text{ \rm closed sets in } [0,1] \} \to [0,\infty]$, such that
   $\Psi(A)$  is a random variable for each closed set $A \subset [0,1]$. Suppose that $\Psi$ satisfies:
\begin{enumerate}
\item
\label{independence}
For $I_1, \cdots, I_j$ disjoint closed intervals, $\Psi(I_1), \cdots, \Psi(I_j)$ are independent random variables.
\item
\label{nonatomic}
$\P\left(\forall x \in [0,1]: \Psi\{x\}=0\right)=1$.

\item
\label{monotonicity}
$\p(A) \leq \p(B)$ a.s., whenever $A \subset B$ and they are both closed.

\item
\label{subadditive}
With probability $1$ $\Psi(\omega, \cdot)$ satisfies the following:
If $(A_i)_i$ is a sequence of closed sets such that $\p(A_j)=0, \forall j$, and $\bigcup_{j=1}^{\infty} A_j$ is a closed set, then $\p\left(\bigcup_{j=1}^{\infty} A_j\right) =0$.
\end{enumerate}
Then $\P\left(\p\left([0,1]\right) >0 \right) \in \{0,1\}$.
\end{theorem}

\begin{proof}[\textbf{Proof}]
Let $\mathcal{D}_n$ denote the set of dyadic intervals of level $n$, i.e.
\[
\mathcal{D}_n = \left\{\left[\frac{k-1}{2^n}, \frac{k}{2^n}\right]: k=1,\cdots,2^n \right\}.
\]
We declare an interval $I \in \mathcal{D}_n$ to be \textbf{good} for the particular $\omega$, if $\p(I)>0$. Define $p_I = \P(\p(I)>0)$. Let $Z_n$ denote the number of good intervals of level $n$, i.e.
\[
Z_n = \sum_{I \in \mathcal{D}_n} \1(\Psi(I)>0).
\]
Then
$\E(Z_n) = \sum_{I \in \mathcal{D}_n} p_I$. From condition \ref{subadditive}, we see that $Z_n$ is an increasing sequence, thus $\E(Z_n)$ is increasing and hence converges. There are two possibilities for the limit, it is either infinite or finite.
\begin{itemize}
\item If $\E(Z_n) \uparrow \infty$, i.e. $\sum_{I \in \mathcal{D}_n} p_I \uparrow \infty$ as $n \to \infty$, then we define
\[
\mathcal{D}'_n = \left\{\left[\frac{k-1}{2^n}, \frac{k}{2^n}\right]: k=1,\cdots,2^n, k \text{ is odd}  \right\}
\]
and
\[
\mathcal{D}''_n = \left\{\left[\frac{k-1}{2^n}, \frac{k}{2^n}\right]: k=1,\cdots,2^n, k \text{ is even}  \right\}.
\]
Since $\sum_{I \in \mathcal{D}_n} p_I \uparrow \infty$, at least one of the two sequences $\sum_{I \in \mathcal{D}'_n} p_I$ or $\sum_{I \in \mathcal{D}''_n} p_I$ converges to infinity.

Assuming wlog that $\sum_{I \in \mathcal{D}'_n} p_I \to \infty$, we have that
\begin{align*}
\P(\p([0,1])>0) = \P\left(\p\left(\bigcup_{I \in \mathcal{D}_n}I\right)>0\right) = \P(\exists I \in \mathcal{D}_n \text{ s.t. } \p(I)>0)\\ \ge
\P(\exists I \in \mathcal{D}'_n \text{ s.t. } \p(I)>0) =
1- \prod_{I \in \mathcal{D}'_n}(1-p_I),
\end{align*}
where to get the second equality we used conditions \ref{monotonicity} and \ref{subadditive} and to get the last equality we used the independence assumption \ref{independence}, since the intervals in $\mathcal{D}'_n$ are disjoint. But $\sum_{I \in \mathcal{D}'_n} p_I \to \infty$, so we get that $\P\left(\p([0,1])>0\right)=1$.

\item
If $E(Z_n) \uparrow C$, where $C$ is a finite positive constant, then we will show that $\P(\p([0,1])>0)=0$. We now declare a point of $[0,1]$ to be \textbf{good} for a particular $\omega$, if all dyadic intervals that contain it are good.

It is easy to see that if for a realization $\omega$, there are at least $k$ good points, then there exists $n_0$ such that $Z_{n_0}(\omega) \geq k$. So, if there is an infinite number of good points, then  $Z_n \to \infty$ as $n \to \infty$. Let $S(\omega)$ denote the set of good points for the realization $\omega$. Then we have that
\[
\E(Z_n) \geq \E(Z_n \1(|S|= \infty)).
\]
If $\P(|S|= \infty)>0$, then from the discussion above and using monotone convergence we get that
$E(Z_n) \uparrow \infty$, which contradicts our assumption that $E(Z_n) \uparrow C$. Hence a.s. there are only finitely many good points.

So, for a.a. $\omega$, we have that $|S(\omega)| < \infty$. Take such an $\omega$. Then we write $\mathcal{D} = \cup_{n}\mathcal{D}_n$ and we have
\begin{align*}
[0,1] = S(\omega) \bigcup \bigcup_{\substack{I \in \mathcal{D} \\ \p(I)=0}}I,
\end{align*}
because if an interval $I$ has no good points, then $\Psi(I)=0$. Indeed, assuming the contrary, i.e. that
$\Psi(I)>0$, then we would get a decreasing sequence $I_n$ of closed intervals with $\Psi(I_n)>0$ and of lengths converging to $0$. But since the space is complete, the intersection of these closed intervals would have to be non-empty, and hence we would obtain a good point.

Thus, since by condition \ref{nonatomic} we have that $\p(S(\omega))=0$, then using also condition \ref{subadditive} we get that $\p([0,1])=0$. Hence we showed that in this case $\p([0,1])=0$ a.s.
\end{itemize}
\end{proof}

\begin{remark}\rm{
Let $\mu$ be the counting measure for the Poisson process on $[0,1]$. Then it satisfies all conditions  of Theorem~\ref{0-1}, except for (\ref{nonatomic}), yet there is no 0-1 law. }
\end{remark}

\begin{corollary}
Let $X$ be a continuous process on $[0,1]$ that satisfies $\P(A|B) \ge c \P(A)$, for all $A \in \sigma\{X(s)-X(t): s \ge t\}$ and $B \in \sigma\{X(s): s\le t\}$ for some constant $c<1$, and let $f$ be a continuous function. Then it satisfies again a 0-1 law, namely $\P(\mathcal{L}((X+f)[0,1])>0) \in \{0,1\}$.
\end{corollary}

\begin{proof}[\textbf{Proof}]
We divide the unit interval in dyadic subintervals in the same way as before and declare an interval $I$ good if $\mathcal{L}((X+f)(I))>0$ and let $Z_n$ be the total number of good intervals of level $n$. Then there are again two possibilities, either $E(Z_n) \to \infty$ or $E(Z_n)\to C$, where $C$ is a finite positive constant. In the second case, everything follows in the same way as in the proof of Theorem \ref{0-1}. Now, if $E(Z_n) \to \infty$, i.e. in the same notation as above
$\sum_{I \in \mathcal{D}_n}p_I \to \infty$, then we have
\begin{align*}
\P(\mathcal{L}((X+f)[0,1])>0) = \P(\exists I \in \mathcal{D}_n : \mathcal{L}((X+f)(I))>0) \\ = 1 - \P(\forall I \in \mathcal{D}_n, \mathcal{L}((X+f)(I))=0).
\end{align*}
But $\P(\forall I \in \mathcal{D}_n, \mathcal{L}((X+f)(I))=0) \le \prod_{I \in \mathcal{D}_n}(1 - cp_I)$, using the assumption on the conditional probabilities.
But $\prod_{I \in \mathcal{D}_n}(1 - cp_I) \to 0$ as $n \to \infty$, since $\sum_{I \in \mathcal{D}_n}p_I \to \infty$, hence we get that
\begin{align*}
\P(\mathcal{L}((X+f)[0,1])>0) =1.
\end{align*}
\end{proof}

Using Theorem \ref{0-1}, we can easily deduce Theorem~\ref{bmdrift} stated in the Introduction.

\begin{proof}[\textbf{Proof of Theorem \ref{bmdrift}}]
For part \textbf{(a)} we define $\Psi(I) = \mathcal{L}(B+f)(A \cap I)$ and then we use Theorem \ref{0-1}.

For part \textbf{(b)} we define $\Psi(I) = \1(((B+f)(A \cap I))^o \neq \emptyset)$. The only condition that requires some justification is \ref{subadditive}.
\newline
Let $A_i$ be a sequence of closed sets with closed union and such that $((B+f)(A_i))^o = \emptyset, \forall i$. Then $((B+f)(\cup A_i))^o = (\cup(B+f)(A_i))^o$, and since the sets $A_i$ are closed and $B+f$ is continuous, then
$\forall i, (B+f)(A_i)$ is also a closed set. From Baire's theorem we get that the interior of $(\cup(B+f)(A_i))$ is empty, since otherwise one of the sets would have nonempty interior.

To prove part \textbf{(c)}, let $\Psi(I) = \1(\text{dim}(B+f)(I)>c)$. Then it is easy to see that $\Psi$ satisfies conditions \ref{independence}, \ref{nonatomic} and \ref{monotonicity}. Finally condition \ref{subadditive} follows from the countable stability property of Hausdorff dimension.
\end{proof}

\begin{remark}{\rm
Note that it is crucial that the function $f$ in Theorem \ref{bmdrift} be deterministic. If $f$ is an adapted continuous function, then the 0-1 law could fail. Here is an example:
Let $B_t=(B^1_t,B^2_t)$ be a two dimensional Brownian motion started from the origin. Let $A_t$ be a process defined as follows: for $t \leq \frac{1}{2}$, define $A_t =0$ and for $t>\frac{1}{2}$,
if $B^1_{\frac{1}{2}}>0$, then let $A_t$ be the modified Hilbert curve as defined in \cite{Tonci}, otherwise set it to 0. In \cite{Tonci} it is proven that $B+A$ is space filling. Then we have that
\[
\P(\mathcal{L}(B+A)[0,1]>0) = \P(B^1_{\frac{1}{2}}>0) = \frac{1}{2}.
\]
}
\end{remark}

We now generalize Theorem \ref{0-1} to any Polish space $X$.

\begin{theorem}
\label{polish}
Let $(\Omega, \mathcal{F},\P)$ be a probability space. Let $X$ be a Polish space, i.e. a complete separable metric space.  Let $\Psi$ be a random set function, i.e. $\Psi: \Omega \times \{\text{compact sets of } X \} \to [0,\infty]$ such that for each compact subset $A$ of $X$, we have that $\Psi(A)$ is a random variable and it satisfies the following:

\begin{enumerate}
\item
\label{independence1}
$\Psi(I_1), \cdots, \Psi(I_j)$ are independent random variables for $I_1, \cdots, I_j$ disjoint closed balls,

\item
\label{nonatomic1}
$\P(\forall x \in X: \p\{x\}=0)=1$.

\item
\label{monotonicity1}
$\p(A) \leq \p(B)$ a.s., whenever $A \subset B$ and they are both compact and

\item
\label{subadditive1}
With probability $1$ $\Psi(\omega,\cdot)$ satisfies: if $(A_i)_i$ is a sequence of compact sets such that $\p(A_j)=0 \quad \forall j$, and $\bigcup_{j=1}^{\infty} A_j$ is a compact set, then $\p\left(\bigcup_{j=1}^{\infty} A_j\right) =0$.
\end{enumerate}
Then we have that $\P\left(\p\left(X\right) >0 \right) \in \{0,1\}$.
\end{theorem}

\begin{proof}[\textbf{Proof}]
Since $X$ is separable there exists a countable dense subset, which we denote by $D$. Let $\Gamma = \{\bar{B}(x,r): x \in D, r \in \Q \}$. For a particular $\omega$, we declare a closed ball $B$ to be \textbf{good} if $\p(B)>0$. A subset $\Lambda$ of $\Gamma$ is of type 1 if it contains only disjoint sets. For such a $\Lambda$ we define $Z_{\Lambda}$ to be the number of good balls in $\Lambda$. Then
$E(Z_{\Lambda}) = \sum_{B \in \Lambda} \P(\p(B)>0)$ and there are two possibilities.
\begin{itemize}
\item
If $\sup_{\Lambda \subset \Gamma: \text{ type 1}} \sum_{B \in \Lambda} \P(\p(B)>0) = \infty$, then if $\Lambda$ is of type 1, we get
\begin{align*}
\P(\p(X)>0) \ge \P(\exists B \in \Lambda: \p(B)>0) = 1 - \prod_{B \in \Lambda}(1-\P(\p(B)>0)) \\
\geq 1-\exp(-\sum_{B \in \Lambda}\P(\Psi(B)>0)).
\end{align*}
Taking $\sup$ over all $\Lambda \subset \Gamma$ of type 1, we deduce that $\P(\p(X)>0) = 1$.

\item
If $\sup_{\Lambda \subset \Gamma: \text{ type 1}} \sum_{B \in \Lambda} \P(\p(B)>0) < \infty$, then again for an $\omega \in \Omega$, we declare a point of $X$ to be good, if every ball in $\Gamma$ that contains it is good. Note that if there are at least $k$ good points, then there must exist a family $\Lambda$ of type 1, such that $Z_{\Lambda} \ge k$.
Let $S(\omega)$ denote the set of good points. On the event that $|S(\omega)| = \infty$, we can find a sequence $(\Lambda_n)_n$ such that $Z(\Lambda_n) \uparrow \infty$. Thus, if $\P(|S| = \infty)>0$, then
by monotone convergence we get that
$\sup_{\Lambda \subset \Gamma: \text{ type 1}} \sum_{B \in \Lambda} \P(\p(B)>0) = \infty$, contradiction. Hence there is only a finite number of good points, $S$, and we now decompose $X$ as
\[
X = S \bigcup \bigcup_{B \in \Gamma: B \cap S = \emptyset}B.
\]
For any $B$ that does not contain any good points we have that $\p(B)=0$, since if $\p(B)>0$, then we could cover $B$ by a countable number of balls of radius $\frac{1}{2}$ and one of them would have $\p(B_1)>0$ and continuing in the same way we would obtain a decreasing sequence of closed balls of radii tending to 0 and thus since the space is complete, this intersection would have to be nonempty, hence we would obtain a good point.

Finally using conditions \ref{nonatomic1} and \ref{monotonicity1} we get that  a.s. $\p(X) = 0$.
\end{itemize}
\end{proof}

\subsection{An event of intermediate probability}\label{counterex}
Here we present an event that seems similar to those discussed in Theorem, yet does not obey a 0-1 law. This example was mentioned in    \cite[Exercise 9.8]{MP}, but there the proof is only sketched. We include it here with more details.

We first recall the definition of the capacity of a set.
\begin{definition}\rm{
Let $K:\R^d\times \R^d \to [0,\infty]$ be a kernel and $A$ a Borel set in $\R^d$. The $K$-energy of a measure $\mu$
is defined to be
\[
\mathcal{E}_K(\mu) = \int \int K(x,y) d\mu(x)d\mu(y)
\]
and the $K$-capacity of $A$ is defined as
\[
\mathrm{Cap}_K(A) = [\inf\{\mathcal{E}_K(\mu): \mu \text{ a probability measure on } A \}]^{-1}.
\]
When the kernel has the form $K(x,y) = |x-y|^{-\alpha}$, then we write
$\mathcal{E}_\alpha(\mu)$ for $\mathcal{E}_K(\mu)$ and
$\mathrm{Cap}_\alpha(A)$
for $\mathrm{Cap}_K(A)$ and we refer to them as the $\alpha$-energy of $\mu$ and the Riesz $\alpha$-capacity
of $A$ respectively.}
\end{definition}

We recall the following theorem (its proof can be found in \cite[Theorem~4.32]{MP} for instance), which gives the connection between the Hausdorff dimension and the Riesz $\alpha$-capacity, because it will be used extensively in this paper.

\begin{theorem}[Frostman]\label{thm:dimcap}
For any closed set $A \subset \R^d$,
\[
\dim A = \sup \{\alpha: \mathrm{Cap}_\alpha(A) >0 \}.
\]
\end{theorem}

\begin{example}
Let $(B_t, 0\le t \le 3)$ be a standard Brownian motion in one dimension. Let $A$ be a closed subset of $[0,3]$. Then we will show that
it is not always true that $\P(B \text{ is 1-1 on } A) \in \{ 0,1 \}$ .
\end{example}

\begin{proof}[\textbf{Proof}]
It is clear that if $A$ is any closed interval, then the above probability is 0.

We are going to use the following equivalence (\cite[Theorem~1.1]{Koshnevisan}): for any two disjoint closed sets $\Lambda_0$ and $\Lambda_1$,
\begin{align}\label{eq:koshncrit}
\P(B(\Lambda_0) \cap B(\Lambda_1) \neq \emptyset) >0 \Leftrightarrow \text{Cap}_{\frac{1}{2}}(\Lambda_0 \times \Lambda_1) >0.
\end{align}
Let
\[
A_0 = \{ \sum_{n=1}^{\infty} \frac{x_n}{2^n}: x_n \in \{0,1\} \text{ and } \forall k, \forall n \in ((2k)!, (2k+1)!], x_n=0 \}
\]
be a closed subset of $[0,1]$ and let
\[
A_1 = \{ 2+ \sum_{n=1}^{\infty} \frac{x_n}{2^n}: x_n \in \{0,1\} \text{ and } \forall k, \forall n \in ((2k-1)!, (2k)!], x_n=0\}
\]
be a closed subset of $[2,3]$.

We will show that $\P(B \text{ is 1-1 on } A_0\cup A_1) \notin \{0,1\}$. This probability is equal to
\[
\P(B(A_0) \cap B(A_1) =\emptyset, B \text{ is 1-1 on } A_0, B \text{ is 1-1 on } A_1).
\]

We have that $\P(B(A_0) \cap B(A_1) \neq \emptyset)$ is strictly smaller than 1, since $A_0 \subset [0,1]$ and $A_1 \subset [2,3]$ and with positive probability the images of $[0,1]$ and $[2,3]$ under Brownian motion are disjoint.

It is easy to see  that $\text{dim}(A_0)=0$.
Indeed, for each $k$, we can cover $A_0$ by at most $2^{(2k)!}$ dyadic intervals of length $2^{-(2k+1)!}$. Hence, for all $k$,
\[
\mathcal{H}^{\alpha}_{\infty}(A_0) \leq 2^{(2k)!} (2^{-(2k+1)!})^{\alpha} \to 0, \text{ as } k \to \infty.
\]
The same argument also gives that $\dim (A_0 \times A_0)=0$ and $\dim (A_1 \times A_1)=0$.

Note that $A_0+A_1 = \{ x+y: x \in A_0, y \in A_1\} = [2,3] $ and since the mapping $f: A_0 \times A_1 \to [2,3]$ given by $f(x,y)= x+y$
is Lipschitz with constant 1, we get that $\text{dim}(A_0 \times A_1) \geq 1$, and hence by Theorem~\ref{thm:dimcap}, $\text{Cap}_{\frac{1}{2}}(A_0 \times A_1) >0$, so from \eqref{eq:koshncrit} we deduce that
\[
\P(B(A_0) \cap B(A_1) \neq \emptyset)>0.
\]

We now want to show that $\P(B \text{ is 1-1 on } A_0) = 1$ and similarly for $A_1$.
\newline
Let $\Delta = \{(x,x):x \in \R \}$ and note that we can write
\[
A_0 \times A_0 \setminus \Delta= \bigcup_{\substack {a,b,c,d \in \Q \\ [a,b]\cap[c,d]=\emptyset}} ([a,b]
\cap A_0) \times ([c,d]\cap A_0).
\]
We now have
\begin{align*}
\P(B \text{ not 1-1 on } A_0) = \P(\exists \quad (s,t) \in A_0 \times A_0 \setminus \Delta : B(s)=B(t))
\\\leq \sum_{\substack {a,b,c,d \in \Q \\ [a,b]\cap[c,d]=\emptyset}}
\P(B([a,b]
\cap A_0) \cap B([c,d]\cap A_0) \neq \emptyset).
\end{align*}
But $[a,b] \cap A_0$ and $[c,d]\cap A_0$ are disjoint and closed sets, so from \eqref{eq:koshncrit} we get that
\[
\P(B([a,b]\cap A_0) \cap B([c,d]\cap A_0) \neq \emptyset)=0,
\]
since $\dim (([a,b]\cap A_0)\times ([c,d]\cap A_0))=0$.


Therefore we deduce that $\P(B \text{ is 1-1 on } A_0\cup A_1) = \P(B(A_0)\cap B(A_1)=\emptyset)$, and hence strictly between 0 and 1.

%
%
%
%
%
\end{proof}

\section{Hitting probabilities}\label{hitting}
In this section, we are going to give the proof of Theorem \ref{interequi}. We will need a preliminary Proposition, given just below, which compares the Green kernels for $B$ and $B+f$. Recall that the transition density of the Brownian motion in $d$ dimensions is
\[
p(t,x,y) = \frac{1}{(2\pi t)^{\frac{d}{2}}} e^{-\frac{|x-y|^2}{2t}}
\]
and the corresponding Green kernel is
\[
G(x,y) = \int_{0}^{\infty} p(t,x,y)\,dt.
\]
Similarly the transition density for $B+f$ is given by
\[
\tilde{p}(t,x,y) = p(t,x-f(0),y-f(t))
\]
and the corresponding Green kernel is
given by $\tilde{G}(x,y)= \int_{0}^{\infty} \tilde{p}(t,x,y)\,dt$.

In dimension $2$ we consider the Green kernel for the killed process $G_{\lambda}(x,y) = \int_{0}^{\infty} e^{-\lambda t}p(t,x,y)\,dt$
and define $\tilde{G}_{\lambda}$ analogously.

\begin{proposition} \label{equivgreen}
Let $(B_t)_t$ be a transient Brownian motion in $d\geq 2$ dimensions. In dimension 2 we kill the Brownian motion after an independent Exponential time of parameter $\lambda$.  Let $f$ be a H\"older$(\frac{1}{2})$ continuous deterministic function $f: \R_+ \to \R^d$ with H\"older constant $K$ and let $\tilde{G}(x,y)$ be the Green kernel of the process $B+f$.
\begin{description}
  \item[(a)] \quad If $d\geq 3$, then there exist positive constants $c_1$ and $c_2$ depending only on $d$ and $K$, such that for all $x,y \in \R^d$, we have $c_1 G(x,y) \le \tilde{G}(x,y) \le c_2 G(x,y)$.
  \item[(b)] \quad If $d=2$, then for all $C>0$, there exist positive constants $c_1$ and $c_2$, depending on $C$, $K$ and $\lambda$, such that for all $x,y \in \R^2$ such that $|x-y|\leq C$, we have that \newline
      $c_1 G_{\lambda}(x,y) \le \tilde{G}_{\lambda}(x,y) \le c_2 G_{\lambda}(x,y)$.
\end{description}
\end{proposition}

\begin{proof}[\textbf{Proof of (a)}]
Since $|x-y+f(t)-f(0)| \geq ||x-y|-|f(t)-f(0)||$, we have that
\begin{align*}
\tilde{G}(x,y) \leq \int_{0}^{\infty} \frac{1}{(2\pi t)^{\frac{d}{2}}} e^{-\frac{|x-y|^2}{2t}} e^{\frac{|x-y||f(t)-f(0)|}{t}} e^{-\frac{|f(t)-f(0)|^2}{2t}} dt,
\end{align*}
and since $f$ is H\"older$(\frac{1}{2})$ we have that $|f(t) -f(0)| \leq K \sqrt{t}$, for all $t$ for a positive constant $K$, and setting $r= |x-y|$ we get
\begin{align*}
\tilde{G}(x,y) \leq \int_{0}^{\infty} \frac{1}{(2\pi t)^{\frac{d}{2}}}e^{-\frac{r^2}{2t}}e^{\frac{Kr}{\sqrt{t}}}dt \le \int_{0}^{t_0}\frac{1}{(2\pi t)^{\frac{d}{2}}}
e^{-\frac{r^2}{4t}} dt + c'\int_{t_0}^{\infty}\frac{1}{(2\pi t)^{\frac{d}{2}}}e^{-\frac{r^2}{2t}}dt,
\end{align*}
where $t_0 = c r^2$ and $c$ is a sufficiently small constant. We also have that
\[
\int_{0}^{\infty} \frac{1}{(2\pi t)^{\frac{d}{2}}}e^{-\frac{r^2}{2t}}dt = c(d) r^{2-d},
\]
(see for instance \cite[Theorem~3.33]{MP}). So,
\[
\int_{0}^{t_0}\frac{1}{(2\pi t)^{\frac{d}{2}}}
e^{-\frac{r^2}{4t}} dt \leq \int_{0}^{\infty}\frac{1}{(2\pi t)^{\frac{d}{2}}}
e^{-\frac{r^2}{4t}} dt= c(d) \left(\frac{r}{\sqrt{2}}\right)^{2-d},
\]
thus there exists a uniform positive constant $c_1$ such that for all $x$ and $y$
\begin{align*}
\tilde{G}(x,y) \leq c_1 G(x,y).
\end{align*}

For the lower bound, using again the H\"older continuity assumption on $f$, we have
\begin{align*}
\tilde{G}(x,y) \geq \int_{0}^{\infty} \frac{1}{(2\pi t)^{\frac{d}{2}}} e^{-\frac{r^2}{2t}} e^{-\frac{|f(t)-f(0)|^2}{2t}} e^{-\frac{r|f(t)-f(0)|}{t}}dt \geq
\tilde{c}\int_{0}^{\infty} \frac{1}{(2\pi t)^{\frac{d}{2}}} e^{-\frac{r^2}{2t}}  e^{-\frac{Kr}{\sqrt{t}}}dt
\geq  c_2 G(x,y),
\end{align*}
for a positive constant $c_2$, uniform over all $x$ and $y$. The last inequality follows in the same way as the upper bound above.
\end{proof}

\begin{proof}[\textbf{Proof of (b)}]
Let $x$ and $y$ satisfy $r=|x-y| \leq C$, for $C>0$. Then for $t_0 = \alpha r^2$, for $\alpha$ a sufficiently small constant, we have that
\begin{align*}
\tilde{G}_\lambda(x,y)\leq \int_{0}^{\infty} \frac{1}{2\pi t} e^{-\frac{r^2}{2t}} e^{\frac{Kr}{\sqrt{t}}} e^{-\lambda t} dt \leq \int_{0}^{t_0} \frac{1}{2\pi t} e^{-\frac{r^2}{4t}} e^{-\lambda t} dt +
e^{K/\sqrt{\alpha}}\int_{t_0}^{\infty} \frac{1}{2\pi t} e^{-\frac{r^2}{2t}} e^{-\lambda t} dt \\ =\int_{0}^{2t_0} \frac{1}{2\pi t} e^{-\frac{r^2}{2t}} e^{-\lambda t} e^{\frac{\lambda t}{2}} dt + e^{K/\sqrt{\alpha}}\int_{t_0}^{\infty} \frac{1}{2\pi t} e^{-\frac{r^2}{2t}} e^{-\lambda t} dt \\ \leq
e^{\lambda C^2\alpha} \int_{0}^{2t_0} \frac{1}{2\pi t} e^{-\frac{r^2}{2t}} e^{-\lambda t}dt
+ e^{K/\sqrt{\alpha}}\int_{t_0}^{\infty} \frac{1}{2\pi t} e^{-\frac{r^2}{2t}} e^{-\lambda t} dt \leq
c_1 G_\lambda(x,y),
\end{align*}
The lower bound follows similarly.
\end{proof}


\begin{proof}[\textbf{Proof of Theorem \ref{interequi} (a)}]
From \cite[Proposition 1.1]{BPP} or \cite[Theorem 8.24]{MP} we have that for any transient Brownian motion
\begin{align} \label{equivmartin}
\frac{1}{2}\text{Cap}_M(A) \leq \P_{x_0}(B \text{ hits } A) \leq \text{Cap}_M(A),
\end{align}
where $M$ is the Martin kernel, defined by $M(x,y) = \frac{G(x,y)}{G(x_0,y)}$, and $A$ is a closed set.
To prove the theorem, it suffices to establish that there exist positive constants $c_1$ and $c_2$ depending only on the H\"older constant of $f$ such that for all starting points $x_0$ and all closed sets $A$
\begin{align}\label{eq:martin}
c_1 \text{Cap}_{M}(A) \leq \P_{x_0}(B+f \text{ hits } A) \leq c_2 \text{Cap}_{M}(A).
\end{align}
The proof of Theorem~8.24 in \cite{MP} can be adapted to show this; we will give the details for this adaptation for the upper bound on the hitting probability.
\newline
Let $\tau = \inf\{t>0: B_t + f(t) \in A\}$ and let $\nu$ be the distribution of $B_\tau+f(\tau)$. The total mass of $\nu$ is $\nu(A) = \P_{x_0}(\tau< \infty)$. From the definition of the Green kernel $\tilde{G}$ for $B+f$ we have that for any $y$
\begin{align}\label{eq:equal}
\E\int_{0}^{\infty} \1(|B_t+f(t) - y| < \epsilon)\,dt = \int_{\mathcal{B}(y,\epsilon)} \tilde{G}(x_0,z)\,dz.
\end{align}
Using Proposition~\ref{equivgreen}, we obtain that there exist constants $c$ and $c'$ that depend only on the H\"older constant of $f$ such that for all $x_0$ and $z$
\begin{align}\label{eq:tilde}
c G(x_0,z)  \leq  \tilde{G}(x_0,z) \leq  c' G(x_0,z).
\end{align}
Integrating over all $t$ the inequality
\[
\P_{x_0}(|B_t+f(t) - y| < \epsilon) \geq \P_{x_0}(|B_t+f(t) - y| < \epsilon, \tau \leq t),
\]
we get
\begin{align*}
\E_{x_0}\int_{0}^{\infty} \1(|B_t+f(t) - y| < \epsilon)\, dt \geq \E_{x_0} \int_{0}^{\infty} \1(B_{t+\tau} + f(t+\tau) -y| < \epsilon) \,dt
\\
= \E_{x_0} \E_{x_0}\left(\int_{0}^{\infty} \1(|B_{t+\tau} + f(t+\tau)|< \epsilon) \,dt \vert\mathcal{F}_{\tau} \right)
= \E_{x_0} \int_A \int_{\mathcal{B}(y,\epsilon)} G_{\tau}(x,z) \,dz \, d\nu(x) ,
\end{align*}
where $G_{\tau}$ is the Green kernel for the process $B_{t+\tau} + f(t+\tau)$, given by 
\[
G_\tau(x,y) = \int_{0}^{\infty} p(t, x - f(\tau), y- f(t+\tau)) \, dt.
\]
Given $\mathcal{F}_\tau$, by the strong Markov property $B_{t+\tau} - B_{\tau}$ is a standard Brownian motion independent of the past and $B_{\tau} + f(t+\tau)$ is a H\"older(1/2) function, independent of $B_{t+\tau} - B_{\tau}$, with the same H\"older constant as $f$. Therefore, given $\mathcal{F}_\tau$, for all $x$ and $z$ we have
\[
c G(x,z) \leq  G_\tau(x,z) \leq c' G(x,z),
\]
for the same constants $c$ and $c'$ appearing in \eqref{eq:tilde}, since they only depend on the H\"older constant of $f$.
\newline
We thus obtain
\[
\E_{x_0}\int_{0}^{\infty} \1(|B_t+f(t) - y| < \epsilon)\, dt  \geq c \int_A \int_{\mathcal{B}(y,\epsilon)}G(x,z) \, dz \, d\nu(x)
\]
and combining that with \eqref{eq:equal} and using \eqref{eq:tilde} we deduce that
\[
c' \int_{\mathcal{B}(y,\epsilon)} G(x_0,z)\,dz \geq c \int_A \int_{\mathcal{B}(y,\epsilon)} G(x,z) \,dz \, d\nu(x).
\]
Dividing through by $\mathcal{L}(\mathcal{B}(0,\epsilon))$ and letting $\epsilon \to 0$ we obtain
\[
c' G(x_0,y) \geq c \int_A G(x,y) \, d\nu(x),
\]
and hence, $\nu(A) \leq \frac{c'}{c} \text{Cap}_M(A)$.
\end{proof}
It is a classical fact that Brownian motion in 2 dimensions is neighborhood recurrent. In the following lemma we will prove that the same is true for $B+f$, if $f$ is a H\"older(1/2) continuous function. 

\begin{lemma}\label{lem:neighbrec}
Let $f$ be a H\"older(1/2) continuous function, $f: \R_+ \to \R^2$ and $B$ a standard Brownian motion in 2 dimensions. Then $B+f$ is neighborhood recurrent. 
\end{lemma}

\begin{proof}[\textbf{Proof}]
Let $D$ denote the unit ball in $\R^2$. Without loss of generality, we will show that $B+f$ hits $D$ infinitely often almost surely. Let $w \in \R^2$. We will prove that
\[
\P_w\left(\bigcap_{n}\{B+f \text{ hits $D$ after time } n\}\right)=1.
\]
The event
$\bigcap_{n}\{B+f \text{ hits $D$ after time } n\}$ is a tail event for Brownian motion, and hence has probability either 0 or 1. Also,
\[
\P_w\left(\bigcap_{n}\{B+f \text{ hits $D$ after time } n\}\right)= \lim_{n \to \infty}\P_w\left(B+f \text{ hits $D$ after time } n\right).
\]
Let $T = \int_{n}^{n^2} \1(|B_t+f(t)|\le 1)\,dt$, which is the time spent in the unit ball between $n$ and $n^2$. Then
\[
\E_w(T) = \int_{n}^{n^2} \P_w(|B_t+f(t)|\le 1) \,dt
\]
and there exist positive constants $c_1$ and $c_2$ that depend only on the H\"older constant of $f$ such that for all $t\geq |w|^2$,
\[
\frac{c_2}{t} \leq \P_w(|B_t+f(t)|\le 1) \leq \frac{c_1}{t}.
\]
We thus obtain that
\[
\E_w(T) \ge c_2 \log{n}.
\]
For the second moment of $T$ we have
\begin{align*}
\E_w(T^2) = 2 \E_w \int_{n}^{n^2}  \int_{s}^{n^2} \1(|B_t + f(t) | \le 1) \1(|B_s + f(s)| \le 1) \,dt \,ds  \\
\le  2 \int_{n}^{n^2} \frac{c_1}{s} \int_{s}^{n^2} \left( \frac{c_1}{t-s} \wedge 1\right) \,dt \,ds
\le
2 \int_{n}^{n^2} \frac{c_1}{s} \int_{0}^{n^2} \left( \frac{c_1}{u} \wedge 1\right) \,du \,ds  \le c (\log{n})^2,
\end{align*}
for a positive constant $c$.
Therefore, applying the second moment method to $T$, namely $\P_w(T>0) \geq \frac{\E_w(T)^2}{\E_w(T^2)}$, we get that $\P_w(T>0) \geq \frac{c_2^2}{c} >0$, and hence
\begin{align*}
\P_w\left(B+f \text{ hits $D$ after time } n\right) \ge \P_w(T>0) >\frac{c_2^2}{c},
\end{align*}
which concludes the proof of the lemma.
\end{proof}

\begin{proof}[\textbf{Proof of Theorem \ref{interequi} (b)}]
From \cite[Proposition 1.1]{BPP} or \cite[Theorem 8.24]{MP} we have that 
\begin{align} \label{equivmartin1}
\frac{1}{2}\text{Cap}_{M_\lambda}(A) \leq \P_{x_0}(B \text{ hits } A) \leq \text{Cap}_{M_\lambda}(A),
\end{align}
where $M_{\lambda}$ is the Martin kernel for the killed Brownian motion $B^\lambda$ defined by $M_\lambda(x,y) = \frac{G_\lambda(x,y)}{G_\lambda(x_0,y)}$.
To prove the first part of the theorem it suffices to establish that for any bounded open set $U$, there exist positive constants $c_1$ and $c_2$ that depend only on $U$, on the H\"older constant of $f$ and on $\lambda$ such that for all $x_0 \in U$ and all closed sets $A\subset U$
\begin{align}\label{eq:martin2d}
c_1 \text{Cap}_{M_\lambda}(A) \leq \P_{x_0}((B+f)^{\lambda} \text{ hits } A) \leq c_2 \text{Cap}_{M_\lambda}(A).
\end{align}
The proof of that follows in the same way as the proof of \eqref{eq:martin} using also Proposition~\ref{equivgreen}(part {\textbf{(b)}}).

For the second part of the theorem, let $A$ be a nonpolar set for $B$, i.e. $\P_u(B \text{ hits } A)>0$, for all starting points $u$. By neighborhood recurrence of Brownian motion, we get that this probability is indeed equal to 1. Let $B^{\lambda}$ and $(B+f)^{\lambda}$ denote the processes killed after an Exponential time of parameter $\lambda$ independent of the Brownian motion. Let $x \in \R^2$. We can find a small $\lambda$ such that
\[
\P_x(B^{\lambda} \text{ hits } A)>c \quad \text{ and } \quad \P_0(B^{\lambda} \text{ hits } A)>c,
\]
for a positive constant $c$.
Using \eqref{eq:martin2d} we deduce that also $\P_x((B+f)^\lambda \text{ hits } A) >0$, and hence 
\[
\P_x(B+f \text{ hits } A) >0.
\]
We now need to show that $\P_x(B+f \text{ hits } A)=1$.
\newline
The event $\{B+f \text{ hits } A \text{ i.o.} \}$ is a tail event, and hence has probability either 0 or 1. Without loss of generality we assume that $A$ is separated from the ball of radius 2 centered at the origin.
By \eqref{eq:martin2d} for a ball $C$ of radius $\frac{1}{2}$ around 0, we can find positive constants $c_3$ and $c_4$ that depend only on $C$ and on the H\"older constant of $f$ , such that for all $z$ in $C$ we have that
\begin{align} \label{d2case}
c_3 \P_{z}(B^{\lambda} \text{ hits } A) \leq \P_{z}((B+f)^{\lambda} \text{ hits } A)\leq c_4 \P_{z}(B^{\lambda} \text{ hits } A).
\end{align}
First we will show that
\begin{align}\label{eq:inf}
\inf_{y \in C} \P_y(B^{\lambda} \text{ hits } A) > c_5>0,
\end{align}
and hence, using \eqref{d2case}, we will get that
\begin{align}\label{lowerbound}
\inf_{y \in C} \P_y(B+f \text{ hits } A) >c_6>0.
\end{align}
To show \eqref{eq:inf}, we will show that for all $x_0 \in C$
\begin{align}\label{eq:zero}
\P_{x_0}(B^{\lambda} \text{ hits } A) \ge c_7 \P_0(B^{\lambda} \text{ hits } A),
\end{align}
for a positive constant $c_7$.\newline
The probability $\P_0(B^{\lambda} \text{ hits } A)$ is bounded from above by the probability that a Brownian motion without killing started from 0 hits the boundary of $\mathcal{B}(x_0,1)$, denoted by $\partial \mathcal{B}(x_0,1)$, where $\mathcal{B}(x_0,1)$ is the ball of radius 1 centered at $x_0 \in C$, and then starting from the hitting point an independent Brownian motion with Exponential($\lambda$) killing hits $A$. Using Poisson's formula (it can be found in \cite[Theorem~3.44]{MP} for instance) we obtain
\begin{align*}
\P_0(B^{\lambda} \text{ hits } A) \le \int_{\partial \mathcal{B}(x_0,1)} \frac{1-|x_0|^2}{|z|^2} \P_z(B^{\lambda} \text{ hits } A)d\varpi(z) \le 4 \int_{\partial \mathcal{B}(x_0,1)} \P_z(B^{\lambda} \text{ hits } A)d\varpi(z),
\end{align*}
where $\varpi$ stands for the uniform distribution on the sphere $\partial\mathcal{B}(x_0,1)$
We also have
\begin{align*}
\P_{x_0}(B^{\lambda} \text{ hits } A) = \int_{\partial \mathcal{B}(x_0,1)} \P_{x_0}(T_{\partial \mathcal{B}(x_0,1)} \le T(\lambda)) \P_z(B^{\lambda} \text{ hits } A)d\varpi(z) \\
= c_8 \int_{\partial \mathcal{B}(x_0,1)}
\P_z(B^{\lambda} \text{ hits } A)d\varpi(z) \geq \frac{c_8}{4}\P_0(B^{\lambda} \text{ hits } A),
\end{align*}
where $T(\lambda)$ is the Exponential killing time of parameter $\lambda$ and $c_8 = \P_y(T_{\partial \mathcal{B}(y,1)} \le T(\lambda))$, a constant independent of $y$.

Let $T_n = \inf\{t \ge n: B(t)+f(t) \in C \}$. By the neighborhood recurrence, Lemma~\ref{lem:neighbrec}, we get that $T_n < \infty$ a.s.\ and thus we have
\begin{align*}
\P_x(B+f \text{ hits } A \text{ after time $n$}) \ge \P_x(B+f \text{ hits } C \text{ after time $n$, then hits $A$}) \\
\ge \inf_{y \in C} \P_y(B(t+T_n)+f(t+T_n) \text{ hits $A$}) = \inf_{y \in C} \P_y(\tilde{B}(t)+\tilde{f}(t) \text{ hits $A$}),
\end{align*}
where $\tilde{B}(t) = B(t+T_n) - B(T_n)$ is by the strong Markov property a standard Brownian motion and $\tilde{f}(t) = f(t+T_n) - f(T_n) + y$, which is still H\"older(1/2) continuous with the same constant as $f$ and conditioned on $T_n$, $\tilde{f}$ is independent of $\tilde{B}$. Hence, from \eqref{lowerbound},
since the constant $c_6$ depends only on the H\"older constant of $f$, we finally get that
\[
\P_x(B+f \text{ hits } A \text{ after time $n$}) \ge c_9>0,
\]
therefore
\[
\P_x(B+f \text{ hits $A$ i.o.}) =1.
\]
The last part of the Theorem, namely that polar sets for Brownian motion are also polar sets for $B+f$ follows easily from equation \eqref{d2case}, and this completes the proof of the theorem.
\end{proof}

\section{Double points}\label{double}
In this section, we give the proof of Theorem \ref{doublepoints}, which was stated in the Introduction. Some of the proofs in this section are similar to those for Brownian motion without drift.

\begin{theorem}\label{inters}
Let $B_1$ and $B_2$ be two independent Brownian motions in $d$ dimensions and let $f_1$ and $f_2$ be two H\"older$(\frac{1}{2})$ continuous deterministic functions $f_1, f_2: \R_+ \to \R^d$ with the same H\"older constant, $K$.
\begin{description}
   \item[(a)] If $d\ge 4$, then, almost surely, $(B_1+f_1)[0,\infty)$ and $(B_2+f_2)[0,\infty)$ have an empty intersection, except for a possible common starting point.
   \item[(b)] If $d\le 3$, then, almost surely, the intersection of $(B_1+f_1)[0,\infty)$ and $(B_2+f_2)[0,\infty)$ is nontrivial, i.e. contains points other than a possible common starting point.
\end{description}
\end{theorem}

\begin{proof} [\textbf{Proof of (a).}]
First we will show that $\P(B_2[0,\infty) \text{ intersects } (B_1+f_1)[0,\infty))=0$ and then using the intersection equivalence of $B_2$ and $B_2+f_2$ from Theorem \ref{interequi}
we will get the result. Recall from the Proof of Theorem~\ref{interequi} the definition of the Martin kernel $M$.
Conditioning on $B_2[0,\infty)$ and using the independence of $B_1$ and $B_2$, we get from \eqref{eq:martin} that there exist positive constants $c_1$ and $c_2$ such that
\[
c_1 \E\left(\text{Cap}_{M}(B_2[0,\infty))\right) \leq \P(B_2[0,\infty) \text{ intersects } (B_1+f_1)[0,\infty))\leq c_2 \E\left(\text{Cap}_{M}(B_2[0,\infty))\right).
\]
From \cite[Proof of part a, Theorem 9.1]{MP} we have that $\E\left(\text{Cap}_{M}(B_2[0,\infty))\right)=0$, and hence concluding the proof. 
\end{proof}

\begin{proof} [\textbf{Proof of (b).}]
If $d=3$, then, almost surely, $\text{Cap}_{M}(B_2[0,\infty))>0$. We suppose first that $B_1+f_1$ and
$B_2+f_2$ start from different points, the one from 0 and the other one from $x$. We then have
\[
\P_{0,x}((B_2+f_2)[0,\infty) \text{ intersects } (B_1+f_1)[0,\infty))>c(K,|x|)>0,
\]
where $c(K,|x|)$ is a constant that only depends on the H\"older constant $K$ and on the distance between the starting points. This follows from Proposition~\ref{equivgreen}, rotational invariance of Brownian motion and the fact that after a rotation $f$ is still a H\"older($\frac{1}{2}$) function with the same H\"older constant as $f$.
Also, by scaling invariance of Brownian motion and the fact that $\frac{f(\alpha^2t)}{\alpha}$ for any $\alpha \neq 0$ is also H\"older($\frac{1}{2}$) with the same H\"older constant, we get that the probability of intersection is lower bounded by a constant that only depends on the H\"older constant and not on the starting points. Thus we have
\begin{align}\label{lowbd}
\P((B_2+f_2)[0,\infty) \text{ intersects } (B_1+f_1)[0,\infty))>c(K)>0.
\end{align}
Let $q \leq 1-c(K)<1$ be the supremum over the starting points of the  probability that
$(B_1+f_1)[0,\infty)$ and $(B_2+f_2)[0,\infty)$ do not intersect.
Then there exists $t$ big enough so that
\[
\P(B_2(t_2)+f_2(t_2)\neq B_1(t_1)+f_1(t_1), \text{ for all } 0 < t_1, t_2 \le t ) \le q+ \epsilon.
\]
By the Markov property,
\begin{align*}
q \leq \P(B_2(t_2)+f_2(t_2)\neq B_1(t_1)+f_1(t_1), \forall t_1, t_2 \le t ) \times \\
\P(B_2(t_2)+f_2(t_2)\neq B_1(t_1)+f_1(t_1), \forall t_1, t_2 > t ) \le q(q+\epsilon)
\end{align*}
and as $\epsilon>0$ was arbitrary, we deduce that $q\le q^2$, and hence $q=0$.
If $B_1+f_1$ and $B_2+f_2$ start from the same point, then we write
\begin{align*}
\P(B_2(t_2)+f_2(t_2)\neq B_1(t_1)+f_1(t_1), \text{ for all } t_1, t_2 >0 ) \\
= \lim_{t \to 0} \P(B_2(t_2)+f_2(t_2)\neq B_1(t_1)+f_1(t_1), \forall t_1, t_2 > t )=0,
\end{align*}
which follows from the Markov  property of Brownian motion applied to time $t$.

For dimensions $d<3$, we project the 3-dimensional motion on the lower dimensional space to obtain nonempty intersection a.s.
\end{proof}

\begin{proof} [\textbf{Proof of Theorem \ref{doublepoints} (a)}]
We will adapt the proof of \cite[Theorem 9.22]{MP}. Let
\begin{align*}
X_1(t) = B\left(\frac{1}{2}+t\right) + f\left(\frac{1}{2}+t\right) - B\left(\frac{1}{2}\right) - f\left(\frac{1}{2}\right) \text{ and } \\ X_2(t) = B\left(\frac{1}{2}-t\right) + f\left(\frac{1}{2}-t\right) - B\left(\frac{1}{2}\right) - f\left(\frac{1}{2}\right).
\end{align*}
Then, since $X_1$ and $X_2$ are independent, by the independence of the increments of the Brownian motion, we get from Theorem \ref{inters} that $X_1$ and $X_2$ intersect almost surely, thus giving the result.
\end{proof}

\begin{proof} [\textbf{Proof of Theorem \ref{doublepoints} (b)}]
Let $\alpha \in [0,1]$  be a rational number. We will show that almost surely, there exist no times $0 \le t_1 < \alpha < t_2 \le 1$ with $B(t_1)+f(t_1)=B(t_2)+f(t_2)$. Let $X_1$ and $X_2$ be given by
\[
X_1(t) = B(\alpha+t) + f(\alpha+t) - B(\alpha) - f(\alpha) \text{ and } X_2(t) = B(\alpha-t) + f(\alpha-t) - B(\alpha) - f(\alpha).
\]
Then by the independence of the increments of Brownian motion we get that $X_1$ and $X_2$ are independent. By the H\"older assumption on $f$ we get that $f(\alpha+t) -f(\alpha)$ and $f(\alpha-t) - f(\alpha)$ are also H\"older($\frac{1}{2}$) continuous functions. Hence from Theorem \ref{inters} we deduce that $X_1$
and $X_2$ will not intersect almost surely, thus giving the result.
\end{proof}

So far we have shown that when the drift $f$ is a H\"older$(\frac{1}{2})$ function, then $B+f$ has no double points in dimension greater than or equal to 4 and has double points in dimension below 4. We will now give an example where adding the drift causes $B+f$ to have double points in dimension greater than 4.

\begin{theorem}\label{fractional}
Let $X$ be a fractional Brownian motion in $d\ge 4$ dimensions with Hurst index $\alpha< \frac{2}{d}$ and let $B$ be an independent standard Brownian motion in $d$ dimensions. Then a.s. $B+X$ has double points.
\end{theorem}

\begin{proof}[\textbf{Proof}]
In \cite[Theorem 2]{Kono} it is proven that a fractional Brownian motion in $d$ dimensions with Hurst index $\alpha< \frac{2}{d}$ has double points a.s. K\^ono's proof works to show that $B+X$ has double points. The only thing  we need to check is that the correlation function $r'$ defined below satisfies the same two inequalities as in \cite{Kono}.
Let $a, \delta, L$ be positive numbers and let $s,t,u,v$ be real numbers satisfying:
\begin{align*}
a \leq |s-t| \leq a + 4 \delta, \quad a \leq |u-v| \leq a + 4 \delta \\
\min\{|s-u|,|s-v|,|t-u|,|t-v| \} \geq L.
\end{align*}
We now let
\begin{align*}
&r' = \frac{\E((X_s+B_s-X_t-B_t)(X_u+B_u - X_v - B_v))}{(|s-t|^{\alpha}+\sqrt{|s-t|})(|u-v|^{\alpha}+\sqrt{|u-v|})}
\\&= \frac{|t-u|^{2\alpha} + |t-u| + |s-v|^{2\alpha} + |s-v| - |s-u|^{2\alpha} - |s-u| - |t-v|^{2\alpha} - |t-v|}{2(|s-t|^{\alpha}+\sqrt{|s-t|})(|u-v|^{\alpha}+\sqrt{|u-v|})}.
\end{align*}
It is easy to see that by choosing $L$ large enough compared with $a$ and $\delta$ we get that $|r'| \leq \epsilon$, for $\epsilon>0$. Indeed, define $f(x) = |x-u|^{2\alpha} - |x-v|^{2\alpha}$, for $x < u,v$. Then
\begin{align*}
|r'| \leq c a^{-2\alpha} |f(t) - f(s)| = c' a^{-2\alpha} |t-s| |u-v| |\xi' - \xi|^{2\alpha - 2},
\end{align*}
where $\xi \in (s,t)$ and $\xi' \in (u,v)$, by applying the mean value theorem twice. Since $2\alpha - 2 <0$, we get that $|r'| \leq \epsilon$ for $L$ large enough.


We will now explain how we get the second inequality that $r'$ satisfies, namely that
\begin{align}
\label{corineq}
1-r' \geq c(|s-u|^{2\alpha} + |t-v|^{2\alpha}) a^{-2\alpha},
\end{align}
for a positive constant $c$, when $s,t,u,v$ are as follows:
\begin{align*}
|s-u|\leq 2\delta, \quad |t-u|\leq 2\delta \\
\min\{|s-t|,|s-v|,|t-u|,|u-v| \} \geq a.
\end{align*}
By translation and scaling it suffices to consider the case when $s=0, t=1$ and $v=1+\gamma$ and wlog we assume that $u<\gamma$.
Thus it suffices to show that
\begin{align*}
4((1+\gamma - u)^{\alpha}+ (1+\gamma - u)^{\frac{1}{2}}) - (1-u)^{2\alpha} - 2+2u - (1+\gamma)^{2\alpha}+ u^{2\alpha} + \gamma^{2\alpha} \geq c \gamma^{2 \alpha},
\end{align*}
for a positive constant $c$. Using Taylor expansion to first order terms and using the fact that $\alpha < \frac{2}{d}$, i.e. $2\alpha <1$, and that we can make $\gamma$ as small as we like since it is smaller than $\delta$, we get inequality \eqref{corineq}.
\end{proof}

\section{Hausdorff dimension of the image and graph}\label{hausd}
We will now state a classical result that will be used in the proof of Theorem~\ref{upperbound}. We include its proof here for the sake of completeness.
\begin{lemma}\label{hahnbanach}
Let $(A,d)$ be a compact metric space and $f: A \to \R$ a continuous function. If $\nu$ is a probability measure on $K=f(A)$, then there exists a probability measure $\mu$ on $A$ such that $\nu= \mu \circ f^{-1}$.
\end{lemma}

\begin{proof}[\textbf{Proof}]
Define a linear functional $\Psi$ on the closed subspace $\Upsilon = \{\phi \circ f \vert \phi \in \mathcal{C}(K) \}$ of $\mathcal{C}(A)$ by
$\Psi(\phi\circ f) = \int_K \phi \, d\nu$.
Clearly, $\|\Psi\|=1$.  By the Hahn-Banach theorem, $\Psi$ can be extended to a linear functional $\tilde{\Psi}$ on $\mathcal{C}(A)$ of norm 1.
Since $\tilde{\Psi}(1)=1$, by the Riesz-representation theorem, there exists a probability measure $\mu$ on $A$ such that $\tilde{\Psi}(g) = \int_A g \, d\mu$, for all $g \in \mathcal{C}(A)$.
\end{proof}

\begin{proof}[\textbf{Proof of Theorem \ref{upperbound}}]
We will first show that $\dim (B+f)(A) \geq (2\dim A) \wedge d$ a.s.
Let $\alpha < \dim A \wedge \frac{d}{2}$. Then by Theorem~\ref{thm:dimcap}, there exists a probability measure $\mu$ on $A$ with finite $\alpha$-energy, i.e. $\mathcal{E}_{\alpha}(\mu) < \infty$. We now define a random measure $\tilde{\mu}$ on $(B+f)(A)$ given by $\tilde{\mu}(D) = \mu((B+f)^{-1}(D))$, for any $D\subset(B+f)(A)$. Thus if we show that
\[
\E(\mathcal{E}_{2\alpha}(\tilde{\mu})) = \E \int  \int \frac{d\tilde{\mu}(x) \,d\tilde{\mu}(y)}{|x-y|^{2\alpha}} = \E \int _{0}^{1} \int_{0}^{1} \frac{d\mu(s) \,d\mu(t)}{|B_t + f(t) - B_s - f(s)|^{2\alpha}} < \infty,
\]
then using Theorem~\ref{thm:dimcap} we will conclude that $\mathrm{dim}(B+f)(A) \geq 2 \alpha$ a.s.\ and hence by letting
$\alpha \uparrow \dim A \wedge \frac{d}{2}$ we will get the result.

Applying Fubini we get that
\[
\E(\mathcal{E}_{2\alpha}(\tilde{\mu})) = \int_{0}^{1} \int_{0}^{1} \E \left(\frac{1}{|B_t + f(t) - B_s - f(s)|^{2\alpha}}\right) \,d\mu(s) \,d\mu(t),
\]
so we need to estimate the expectation appearing in the last integral.

Take $t > s$, and write $\beta = \sqrt{t-s}$ and $u = f(s) - f(t)$. Then we want to evaluate
\begin{align*}
\E\frac{1}{|\beta B_1 - u|^{2\alpha}} = \int_{\R^d} \frac{1}{(2 \pi)^{\frac{d}{2}}} \frac{e^{-\frac{|x|^2}{2}}}{|\beta x -u|^{2\alpha}}\,dx = \frac{1}{|\beta|^{2\alpha}} \int_{\R^d} \frac{1}{(2 \pi)^{\frac{d}{2}}} \frac{e^{-\frac{|x|^2}{2}}}{|x -\frac{u}{\beta}|^{2\alpha}}\,dx.
\end{align*}
Define $g(x) = e^{-\frac{|x|^2}{2}} $ and $h(x) = \frac{1}{|x|^{2\alpha}}$. They are both decreasing functions of $|x|$ and hence using that
for all $w \in \R^d$ we have that
\begin{align}\label{eq:cheb}
\int_{\R^d} (g(x) - g(x - w)) (h(x) - h(x - w)) \,dx \geq 0,
\end{align}
we get that
\begin{align}
\label{integ}
\frac{1}{|\beta|^{2\alpha}}  \int \frac{1}{(2 \pi)^{\frac{d}{2}}} \frac{e^{-\frac{|x|^2}{2}}}{|x -\frac{u}{\beta}|^{2\alpha}}\,dx \leq \int \frac{1}{(2 \pi)^{\frac{d}{2}}} \frac{e^{-\frac{|x|^2}{2}}}{|\beta x|^{2\alpha}}\,dx.
\end{align}
Hence we obtain that
\[
\E \frac{1}{|B_t + f(t) - B_s - f(s)|^{2\alpha}} \leq \E \frac{1}{|B_t - B_s|^{2\alpha}}
\]
and thus
\[
\E(\mathcal{E}_{2\alpha}(\tilde{\mu})) \leq \E \int _{0}^{1} \int_{0}^{1} \frac{\,d\mu(s) \,d\mu(t)}{|B_t - B_s|^{2\alpha}}.
\]
The rest of the proof follows in the same way as the proof of McKean's theorem (see for instance \cite[Theorem~4.33]{MP}).

We will now show that $\dim (B+f)(A) \geq \dim f(A)$ a.s. Let $\alpha < \dim f(A)\leq d$.
Then again by Theorem~\ref{thm:dimcap}, there exists a probability measure $\nu$ on $K=f(A)$ with finite $\alpha$-energy, that is $\mathcal{E}_\alpha(\nu) < \infty$.
By Lemma~\ref{hahnbanach}, we can find a probability measure $\tilde{\nu}$ on $A$ such that $\nu = \tilde{\nu} \circ f^{-1}$. Then
\begin{align}
\label{energy}
\int_K \int_K \frac{\,d\nu(x) \,d\nu(y)}{|x-y|^\alpha} = \int_{0}^{1} \int_{0}^{1} \frac{\,d\tilde{\nu}(s) \,d\tilde{\nu}(t)}{|f(s)-f(t)|^\alpha} < \infty.
\end{align}
We finally define a random measure $\mu'$ on $K_f = (B+f)(A)$ by $\mu'(D) = \tilde{\nu}((B+f)^{-1}(D))$, for any $D \subset K_f$.
We want to show that
\[
\E(\mathcal{E}_\alpha(\mu')) < \infty,
\]
and hence using the energy method again we will deduce that $\dim (B+f)(A) \ge \alpha$ a.s.
Finally, just as before, letting $\alpha \uparrow \dim f(A)$ we will obtain that almost surely $\dim (B+f)(A) \geq \dim f(A)$.
\newline
We now have that
\begin{align*}
\E(\mathcal{E}_\alpha(\mu')) = \E \int _{K_f} \int_{K_f} \frac{\,d\mu'(x) \,d\mu'(y)}{|x-y|^\alpha} = \E \int _{0}^{1} \int_{0}^{1} \frac{\,d\tilde{\nu}(s) \,d\tilde{\nu}(t)}{|B_s + f(s) - B_t - f(t)|^\alpha}.
\end{align*}
Take $s>t$ and set $\beta = \sqrt{s-t}$ and $w=f(s)-f(t)$.
\begin{align*}
\E\frac{1}{|B_s + f(s) - B_t - f(t)|^\alpha} = \E\frac{1}{|\beta B_1 + w|^\alpha} = \E \frac{1}{\beta ^\alpha|B_1 + \frac{w}{\beta}|^\alpha} = \frac{1}{\beta^\alpha} \int \frac{1}{(2\pi)^{\frac{d}{2}}|x + \frac{w}{\beta}|^\alpha}  e^{-\frac{|x|^2}{2}}\,dx.
\end{align*}
Let $u= \frac{w}{\beta}$. Then
\begin{align*}
\int_{\R^d} \frac{1}{|x + u|^\alpha}  e^{-\frac{|x|^2}{2}}\,dx = \int_{|x+u|\geq \frac{|u|}{2}}\frac{1}{|x + u|^\alpha}  e^{-\frac{|x|^2}{2}}\,dx + \int_{|x+u|< \frac{|u|}{2}}\frac{1}{|x + u|^\alpha}  e^{-\frac{|x|^2}{2}}\,dx
\\ \leq
\frac{c_1}{|u|^{\alpha}} + e^{\frac{-|u|^2}{4}} \int _{|x|< |u|} \frac{1}{|x|^\alpha} \,dx = \frac{c_1}{|u|^{\alpha}} + c_2 e^{\frac{-|u|^2}{4}} \int_{0}^{|u|} \frac{r^{d-1}}{r^\alpha}\,dr \leq \frac{c_3}{|u|^{\alpha}},
\end{align*}
since $d>\alpha$. Hence,
\begin{align*}
\E\frac{1}{|B_s + f(s) - B_t - f(t)|^\alpha} \leq \frac{C}{|f(t)-f(s)|^\alpha}
\end{align*}
and thus using \eqref{energy} we deduce that
\[
\E(\mathcal{E}_\alpha(\mu')) < \infty.
\]
\end{proof}


\begin{remark}
\rm{
We note that when $f$ is a $\gamma$-H\"older continuous function, then $\dim f(A)\leq \frac{1}{\gamma} \dim A$ and similarly
for Minkowski dimension. Hence using the a.s.\ $\alpha$-H\"older continuity property of $B$ for any $\alpha<\frac{1}{2}$, we deduce that $\mathrm{dim}(B+f)(A) \leq \max\{2,\frac{1}{\gamma}\} \mathrm{dim}A$.
We also note that for any function $f$ we have that almost surely
\begin{align}
\label{dimineq}
\mathrm{dim}(B+f)(A) \leq \mathrm{dim}_M B(A) + \mathrm{dim}f(A) \leq 2 \mathrm{dim}_M A + \mathrm{dim}f(A),
\end{align}
where $\dim_M$ denotes the Minkowski dimension. The first inequality follows from \cite[Theorem~8.10]{Matilla} and
the second from the H\"older property of Brownian motion.
}
\end{remark}

\begin{remark}
\rm{
We note that the lower bound for the Hausdorff dimension of the image of $B+f$ given in Theorem~\ref{upperbound} is sharp in some cases. For example let $X$ be an independent fractional Brownian motion in $d\geq 2$ dimensions of Hurst index $\alpha$. Then it is $\beta$-H\"older continuous for any $\beta < \alpha$ and it is known that almost surely $\dim(\mathrm{Im}X) = \frac{1}{\alpha}\wedge d$ (\cite[Chapter 18]{Kahane}).
Therefore $\dim(\mathrm{Im}(B+X)) = \max\{2,\frac{1}{\alpha}\wedge d\}$ a.s.}
\end{remark}

We now give an example where the lower bound from Theorem~\ref{upperbound} is not sharp and where the upper bound given in \eqref{dimineq} is almost sharp.

\begin{example}
\rm{
Let $B$ be a standard Brownian motion in 3 dimensions and let $f(t)=(f_1(t),0,0)$, where $f_1$ is a fractional Brownian motion independent of $B$ of Hurst index $\alpha$. Then $\mathrm{dim}f[0,1] = 1$ a.s. For $\alpha$ small, we have that almost surely
$\dim (B+f)[0,1] = 3-2\alpha$, which is a special case of \cite[Theorem~1]{Cuzick}.
}
\end{example}

So far we have obtained bounds for the Hausdorff dimension of the image of $B+f$. Similar results to the ones proved for the image hold also for the graph and the proofs use the same methods as before. We will only point out the parts where they differ.

\begin{proof}[\textbf{Proof of Theorem \ref{graphdim} (a)}]
To show that $\dim G(B+f) \geq \frac{3}{2}$, we will adapt the proof of \cite[Theorem~4.29~a]{MP}. Let $\alpha < \frac{3}{2}$. Define a random measure $\mu$ on the graph of $B+f$ by
\[
\mu(A) = \mathcal{L}(\{0\leq t \leq 1: (t,B_t+f(t)) \in A\}).
\]
We want to show that this measure has finite $\alpha$-energy, i.e. that
\[
\E \int  \int \frac{\,d\mu(x) \,d\mu(y)}{|x-y|^\alpha} = \E \int _{0}^{1} \int_{0}^{1} \frac{\,ds \,dt}{(|t-s|^2 + |B_t + f(t) - B_s - f(s)|^2)^\frac{\alpha}{2}} < \infty.
\]
Using again \eqref{eq:cheb} as in the proof of Theorem \ref{upperbound} we get that
\[
\E \frac{1}{(|t-s|^2 + |B_t + f(t) - B_s - f(s)|^2)^\frac{\alpha}{2}} \leq \E\frac{1}{(|t-s|^2 + |B_t - B_s|^2)^{\frac{\alpha}{2}}}.
\]
Now the rest of the proof follows just as in \cite[Theorem 4.29 a]{MP}.

It remains to show that $\dim G(B+f) \geq \dim G(f)$. Let $\alpha < \dim G(f)$. Then there exists a measure $\nu$ on $G(f)$ such that $\mathcal{E}_\alpha(\nu) < \infty$.  Then there exists a measure $\mu$ on $[0,1]$ such that
\[
\nu(A) = \mu(\{t: (t,f(t)) \in A\}).
\]
Next we define a random measure $\mu'$ on $G(B+f)$ by $\mu'(A) = \mu(\{t: (t,B_t + f(t)) \in A\})$. We will show that \newline $\E(\mathcal{E}_\alpha(\mu'))< \infty$, i.e.
\[
\E \int _{0}^{1} \int_{0}^{1} \frac{\,d\mu(s) \,d\mu(t)}{(|t-s|^2+ |B_t + f(t) - B_s - f(s)|^2)^{\frac{\alpha}{2}}} < \infty.
\]
Using the same arguments as in the proof of Theorem \ref{upperbound} we get that
\[
\E \frac{1}{(|t-s|^2 + |B_t + f(t) - B_s - f(s)|^2)^{\frac{\alpha}{2}}} \leq \frac{C}{(|t-s|^2+|f(t) - f(s)|^2)^{\frac{\alpha}{2}}}
\]
and thus the result follows since $\E (\mathcal{E}_{\alpha}(\nu)) < \infty$.
\end{proof}

\begin{proof}[\textbf{Proof of Theorem \ref{graphdim} (b)}]
We have that almost surely $\dim G(B+f) \geq 2$ from Theorem \ref{upperbound}, since the dimension of the graph is always bigger than the dimension of the image. We only need to show that it is greater than the dimension of the graph of $f$. To prove that we use the energy method just as in the proof for the case $d=1$.
\end{proof}

\section*{Question}
Can the zero-one law (Theorem~\ref{bmdrift}) be extended to Gaussian processes $X$ with dependent increments, e.g. fractional Brownian
motion or the Gaussian Free Field in a domain $U\subset \R^d$?
For example, does $\{\mathcal{L}(X+f)(A)>0\}$ satisfy a 0-1 law?

\section*{Acknowledgements}
We thank Ton\'ci Antunovi\'c and Brigitta Vermesi for helpful comments.

\end{document}